\documentclass[11pt,reqno,oneside]{amsart}

\usepackage{graphicx}

\usepackage{pdfsync}

\usepackage{hyperref}

\newtheorem{theorem}{Theorem}[section]
\newtheorem*{theorem*}{Theorem}
\newtheorem{lemma}{Lemma}[section]
\newtheorem{corollary}{Corollary}[section]

\newcommand{\esssup}{\mathop{\mathrm{ess\,sup}}} 
\newcommand{\essinf}{\mathop{\mathrm{ess\,inf}}}

\def\hx{\hat{x}}

\def\hg{\hat{g}}

\def\uu{{\bf u}}

\def\vv{{\bf v}}

\def\gg{{\bf g}}

\def\hvv{\hat{\vv}}

\def\k{\kappa}
\def\O{\Omega}
\def\o{\omega}
\def\ho{\hat{\omega}}
\def\bo{\bar{\omega}}

\def\vfi{\varphi}
\def\vp{\varpi}

\def\di{\mbox{div\,}}

\def\di{\mbox{div\,}}

\def\R{{\mathbb R}}
\def\N{{\mathbb N}}
\def\hO{\hat{\Omega}}
\def\W1p{W^{1,p}(\o)}

\def\divp{\mbox{(div)}_p}

\begin{document}
%\date{\today}
\section*{}

\setcounter{equation}{0}
\title[]
{A decomposition technique for integrable functions\\ with applications to the divergence problem}

\author[F. L\'opez Garc\'\i a]{Fernando L\'opez Garc\'\i a}
\address{Department of Mathematical Sciences\\ Worcester Polytechnic Institute \\ 
100 Institute Road \\ Worcester \\ MA 01609 \\ USA.} \email{flopezgarcia@wpi.edu}



\begin{abstract}
Let $\Omega\subset \mathbb{R}^n$ be a bounded domain that can be written as $\Omega=\bigcup_{t} \Omega_t$, where $\{\Omega_t\}_{t\in\Gamma}$ is a countable collection of domains with certain properties. In this work, we develop a technique to decompose a function $f\in L^1(\Omega)$, with vanishing mean value, into the sum of a collection of functions $\{f_t-\tilde{f}_t\}_{t\in\Gamma}$ subordinated to $\{\Omega_t\}_{t\in\Gamma}$ such that $Supp\,(f_t-\tilde{f}_t)\subset\Omega_t$ and $\int f_t-\tilde{f}_t=0$.  
As an application, we use this decomposition to prove the existence of a solution in weighted Sobolev spaces of the divergence problem $\di\uu=f$  and the well-posedness of the Stokes equations on H\"older-$\alpha$ domains and  some other domains with an external cusp arbitrarily narrow. We also consider arbitrary bounded domains. The weights used in each case depend on the type of domain.
\end{abstract}
\maketitle

\section{Introduction}
\label{intro}

In this paper we show a kind of atomic decomposition for an integral function $f\in L^1(\O)$ if $\O$ is a bounded domain which can be written as the union of a countable collection of domains $\{\O_t\}_{ t\in \Gamma}$ with certain properties. This result is based on a decomposition developed by Bogovskii in \cite{B}, where $\Gamma$ is finite. The goal of this result is to write a function $f$ with $\int f=0$ as the sum of a collection of functions $\{f_t-\tilde{f}_t\}_{ t\in \Gamma}$ such that $Supp\{f_t-\tilde{f}_t\}\subset \O_t$ and $\int_{\O_t} f_t-\tilde{f}_t=0$. As Bogovskii did in his paper we use this decomposition to study the existence of solutions of the divergence problem, and posteriorly the well-posedness of the Stokes equations. 

Let us introduce the divergence problem for a bounded domain $\O\subset\R^n$. Given $f\in L^p(\O)$,  with vanishing mean value and $1 < p < \infty$, the divergence problem deals with the existence of a solution $\uu$ in the Sobolev space $W_0^{1,p}(\O)^n$ of $\di\uu=f$ satisfying
\begin{eqnarray}\label{Introduction div}
\|D\uu\|_{L^p(\O)}\leq C_{\O} \|f\|_{L^p(\O)},
\end{eqnarray}
where $D\uu$ is the differential matrix of $\uu$. This problem has been widely studied and it has many applications, for example, in the particular case $p=2$, it is fundamental for the variational analysis of the Stokes equations (see \cite{G}). It is also well known for its relation with some inequalities such as Korn and Sobolev Poincar\'e. 

Consequently, several methods have been developed to prove the existence of a solution of $\di\uu=f$ satisfying (\ref{Introduction div}) under different assumptions on the domain (see for example \cite{ADM}, \cite{ASV}, \cite{B}, \cite{DRS}, \cite{DM}, \cite{L}).

On the other hand, this result fails if $\O$ has an external cusp or arbitrarily narrow ``corridors", see \cite{ADLg} and \cite{F}.
 However, the existence of solutions of the divergence problem holds in some of these ``bad" domains if we consider weighted Sobolev spaces with an estimate weaker than (\ref{Introduction div}).  A similar analysis can be done for its related results. Since the non-existence of standard solutions arises because of the bad behavior of the boundary, it seems natural to work with weights involving the distance to the boundary of $\O$ or a subset of it. The following are some papers considering the divergence problem or related results in weighted Sobolev spaces \cite{ADL}, \cite{DRS}, \cite{DL1}, \cite{DL2} and \cite{S}.

As we mentioned before there are many different approaches to this problem. In the present paper, as it was done in \cite{DRS} and \cite{DMRT}, 
we use a decomposition of the function $f$ in $\di\uu=f$ to generalize results valid on simple domains, such as rectangles or star-shaped domains, to more general cases.  

\bigskip

The paper is organized in the following way: In Section \ref{preliminaries}, we include some notations and preliminary results. In Section \ref{DecompositionTechnique}, we show the main result of this paper, a decomposition technique for integrable functions defined over a bounded domain $\O$ which is written as the union of a collection of subdomains $\{\O_t\}_{t\in \Gamma}$ with some properties. The set $\Gamma$ is required to have a certain partial order structure. In the following sections we include three different applications of the decomposition developed in Section \ref{DecompositionTechnique}. These sections can be independently read. In section \ref{General domains}, we show the existence of a weighted right inverse of the divergence operator on arbitrary bounded domains. In Sections \ref{Cuspidal domains} and \ref{Holder}, we prove the existence of a solution of the divergence problem and the well-posedness of the Stokes equations on some domains with an external cusp arbitrarily narrow and on bounded H\"older-$\alpha$ domains in $\R^n$. The weights in these two final sections are more specific than the one used in Section \ref{General domains}. More precisely,  the weights are related to the distance to the cusp and to the distance to the boundary of the domain respectively.

\section{Preliminaries and notations}\label{preliminaries}

Let $\O\subset\R^n$ be a bounded domain. Given a measurable positive function $\o:\O\to R_{>0}$ we denote with $L^p(\O,\o)$ the weighted space with norm
\[\|f\|_{L^p(\O,\o)}=\|f\o\|_{L^p(\O)},\]
and with $W^{1,p}_0(\O,\o)$ the weighted Sobolev space defined as the closure of  $C_0^\infty(\O)$ with norm \[\|u\|_{W_0^{1,p}(\O,\o)}=\|Du\|_{L^p(\O,\omega)},\]
where $Du$ is the differential matrix of $u$. Observe that the seminorm $\|Du\|$ results a norm in the trace zero space.

We say that $\O$ satisfies $\divp$, for $1<p<\infty$, with constant $C_\O$ if for any $f\in L^p_0(\O):=\{g\in L^p(\O)\,:\, g\text{ has vanishing mean value }\}$ there is a solution $\uu\in W^{1,p}_0(\O)^n$ of $\di\uu=f$ satisfying (\ref{Introduction div}). We also use $C_A$ to denote a constant depending on $A$, where $A$ could not be a domain.

In the next lemma we compare $C_\O$ with $C_{\hO}$, where $\O$ is a domain obtained by applying an affine function to a domain $\hO$ satisfying $\divp$. This result is standard and the proof uses the Piola transform. 
Before stating the lemma, let us define the $p$-norm of a matrix $A\in R^{n\times n}$ as \[\|A\|_p:=\sup_{\vv\neq {\bf 0}}\frac{\|A\vv\|_p}{\|\vv\|_p},\]
where the $p$-norm of a vector is $\|\vv\|^p_p=\sum_{1\leq i\leq n}|v_i|^p$.

\begin{lemma}\label{rectangles} Let $\hO\subset\R^n$ be a domain satisfying ${\mbox (\di)}_p$ and $F:\R^n\to\R^n$ an affine function defined by $F(\hx)=B\hx+b$, where $B\in R^{n\times n}$ is an invertible matrix and $b\in\R^n$. Then, $\O=F(\hO)$ satisfies $\divp$ with a constant $C_\O$ bounded by \[C_\O\leq \|B\|_p\|(B^{-1})^\prime \|_pC_{\hO}.\]
In particular, $C_{\O}=C_{\hO}$ if $B=\lambda I$, where $I$ is the identity matrix and $\lambda\neq 0$ is a real number.
\end{lemma}

\begin{proof}
In order to simplify the notation we assume that all the vectors are in $\R^{n\times 1}$. Given $f\in L^p_0(\O)$, the function $\hg(\hx)=f(F(\hx))$ belongs to $L^p_0(\hO)$. Thus, we define the vector field $\uu(x):=B\hvv(F^{-1}(x))$, where $\hvv\in W_0^{1,p}(\hO)^n$ is a solution of $\di\hvv=\hg$, with  \[\|\hvv\|_{W_0^{1,p}(\hO)}\leq C_{\hO} \|\hg\|_{L^p(\hO)}.\]  It can be seen that the differential matrix of $\uu$ is $D\uu(x)=B\,D\hvv(F^{-1}(x))\,B^{-1}$, and as the trace is invariant under conjugation we can assert that 
$\di\uu(x)=\di\hvv(F^{-1}(x))=f(x)$.
On the other hand, using change of variables it can be seen that 

\begin{eqnarray*}
\|D\uu\|_{L^p(\O)}&\leq& \|B\|_p\|(B^{-1})^\prime\|_p{\rm det}(B)^{1/p}\|D\hvv\|_{L^p(\hO)}\\
&\leq& \|B\|_p\|(B^{-1})^\prime\|_p{\rm det}(B)^{1/p}C_{\hO}\|\hg\|_{L^p(\hO)}=\|B\|_p\|(B^{-1})^\prime\|_pC_{\hO}\|f\|_{L^p(\O)}.
\end{eqnarray*}
\end{proof}

As we mentioned before an important application of the existence of a right inverse of the divergence operator is the well-posedness of the Stokes equations, given by:
\begin{align}\label{Stokes}
\begin{cases}
-\Delta\uu+\nabla p&=\gg  \quad \text{in $\O$}\\ 
\di\uu&=h  \quad \text{in $\O$}\\
\uu&=0  \quad \text{on $\partial\O$}. 
\end{cases}
\end{align}
For a bounded Lipschitz domain $\O\subset\R^n$ (or more generally a John domain \cite{ADM}), it is known that, if $\gg\in H^{-1}(\O)^n$, the dual space of $H^1_0(\O)^n$, and $h\in L^2(\O)$ with vanishing mean value, there exists a unique variational solution $(\uu,p)$ in $H_0^1(\O)^n\times L^2(\O)$. Moreover, this solution satisfies the a priori estimate:
\begin{eqnarray*}
\|D\uu\|_{L^2(\O)}+\|p\|_{L^2(\O)}\leq C\left(\|\gg\|_{H^{-1}(\O)}+\|h\|_{L^2(\O)}\right),
\end{eqnarray*}
where the constant $C$ depends only on $\O$ .

On the other hand, it is known that this result fails in general for domains with external cusps. However, it was proved in \cite{DL2} that the well-posedness of the incompressible Stokes equations ($h=0$ in (\ref{Stokes})) is valid in weighted Sobolev spaces for an arbitrary bounded domain $\O$ if there exists a standard solution of $\di\uu=f$, where $f$ is in a weighted Sobolev spaces. This result is stated bellow.

\begin{theorem*}\label{Theorem DL2} Let $\omega\in L^1(\O)$ be a positive function. Assume that for any $f\in L^2(\O,\omega^{-1/2})$, with vanishing mean value, 
there exists $\uu\in H_0^1(\O)^n$ such that $\di\uu = f$ and
\[\|D\uu\|_{L^2(\O)} \leq C\|f\|_{L^2(\O,\omega^{-1/2})},\]
with a constant $C$ depending only on $\O$ and $\omega$. Then, for any $\gg\in H^{-1}(\O)^n$, there exists a unique $(\uu, p)\in H_0^1(\O)^n\times L^2(\O,\omega^{1/2})$, with $\int_\O p\omega=0$, incompressible solution of the Stokes problem (\ref{Stokes}). Moreover,
\[\|D\uu\|_{L^2(\O)} + \|p\|_{L^2(\O,\omega^{1/2})} \leq C\|\gg\|_{H^{-1}(\O)},\] 
where $C$ depends only on $\O$ and $\omega$.
\end{theorem*}

\section{A decomposition technique for integrable functions} \label{DecompositionTechnique}

We start this section with an example of Bogovskii's decomposition when $\O$ is a domain written as the union of a collection of subdomains $\{\O_i\}_{0\leq i\leq 2}$. We present the example using our notation. Let $f\in L^p(\O)$ be a function with vanishing mean value. Thus, using a partition of the unity $\{\phi_i\}_{0\leq i\leq 2}$ subordinated to $\{\O_i\}_{0\leq i\leq 0}$ we can write $f$ as: 
\[f=f_0+f_1+f_2=f\phi_0+f\phi_1+f\phi_2.\]
Now, 
\begin{eqnarray*}
f=f_0+\left(f_1+\dfrac{\chi_{B_2}}{|B_2|}\int_{\O_2}f_2\right) + \underbrace{\left(f_2-\dfrac{\chi_{B_2}}{|B_2|}\int_{\O_2}f_2\right)}_{f_2-\tilde{f}_2},
\end{eqnarray*}
where $B_2=\O_2\cap \O_1$. Note that the function $f_2-\tilde{f}_2$ has its support in $\O_2$ and $\int f_2-\tilde{f}_2=0$. Finally, we repeat the process with the first two functions. Thus, if $B_1=\O_1\cap \O_0$ we have that 
\begin{eqnarray}\label{Bogovskii}
f&=&\overbrace{\left(f_0+\dfrac{\chi_{B_1}}{|B_1|}\int_{\O_1\cup\O_2}f_1+f_2\right)}^{f_0-\tilde{f}_0}\\ 
&+&\underbrace{\left(f_1+\dfrac{\chi_{B_2}}{|B_2|}\int_{\O_2}f_2-\dfrac{\chi_{B_1}}{|B_1|}\int_{\O_1\cup\O_2}f_1+f_2\right)}_{f_1-\tilde{f}_1} + \underbrace{\left(f_2-\dfrac{\chi_{B_2}}{|B_2|}\int_{\O_2}f_2\right)}_{f_2-\tilde{f}_2},\notag
\end{eqnarray}
obtaining the claimed decomposition. Note that we have used the vanishing mean value of $f$ only to prove that $f_0-\tilde{f}_0$ integrates zero. If we do not assume any other thing but integrability, we have that $\int f_i-\tilde{f}_i=0$ if $i\neq 0$ and $\int f_0-\tilde{f}_0=\int f$.

In this work we extend the decomposition shown in (\ref{Bogovskii}) when $\O$ is the union of a collection of subdomains $\{\O_t\}_{t\in \Gamma}$, where $\Gamma$ is a partial ordered countable set instead of a totally ordered finite set. In fact, $\Gamma$ is a rooted tree and the partial order is inherited from the graph structure. 

Let us recall some definitions. A {\bf rooted tree} is a connected graph in which any two vertices are connected by exactly one simple path, and a {\bf root} is simply a distinguished vertex $a\in\Gamma$. For these graphs it is possible to define a {\bf partial order} $\preceq$ as  $s\preceq t$ if and only if the unique path connecting $t$ with the root $a$ passes through $s$. Moreover, the {\bf height} or {\bf level} of any $t\in \Gamma$ is the number of vertices in $\{s\in\Gamma\,:\,s\preceq t\text{ with }s\neq t\}$. {\bf The parent} of a vertex $t\in \Gamma$ is the vertex $s$ satisfying that $s\preceq t$ and its height is one unit lesser than the height of $t$. We denote the parent of $t$ by $t_p$. 
It can be seen that each $t\in\Gamma$ different from the root has a unique parent, but several elements on $\Gamma$ could have the same parent.

\subsection{A ``tree'' of domains}\label{def tree}

Our decomposition for functions in $L^1(\O)$ is subordinated to a given decomposition of $\O$, which has to satisfy certain properties stated below. Thus, let $\{\O_t\}_{t\in \Gamma}$ be a countable collection of subdomains of $\O$, where $\Gamma$ is a tree with root $a$, that satisfies the following properties: 
\begin{enumerate}
\item[{\bf (a)}]\label{a} $\displaystyle{\chi_\O(x)\leq \sum_{t\in\Gamma}\chi_{\O_t}(x)\leq N\chi_\O(x)}$, for almost every $x\in \Omega$.
\item[{\bf (b)}]\label{b} For any $t\in \Gamma$ different from the root there exists a set $B_t\subseteq\O_{t}\cap\O_{t_p}$ with no trivial Lebesgue measure. In addition, the collection $\{B_t\}_{t\neq a}$ is pairwise disjoint.
\end{enumerate}

Now, for this collection, we define $\{W_t\}_{t\in \Gamma}$ as $\displaystyle{W_t=\bigcup_{s\succeq t} \O_s}$. Finally, let us denote the characteristic function of $B_t$ by $\chi_t$.

\subsection{A decomposition on a ``tree'' of domains}\label{tree subsection}

Let $\{\phi_t\}_{t\in\Gamma}$ be a partition of the unity subordinated to $\{\O_t\}_{t\in\Gamma}$. Thus, $f$ can be decomposed into $f=\sum_{t\in\Gamma} f_t$, where $f_t=f\phi_t$, and 
\begin{equation*}
\sum_{t\in\Gamma}\|f_t\|^p_{L^p(\O_t)} \leq N \|f\|^p_{L^p(\O)}.
\end{equation*}

Thus,  similarly to (\ref{Bogovskii}), we define $\tilde{f}_t$ for $t\in\Gamma$ as
\begin{eqnarray}\label{tree decomposition}
\tilde{f}_t(x):=\dfrac{\chi_t(x)}{|B_t|}\int_{W_t}\sum_{k\succeq t}f_k-\sum_{s:\,s_p=t}\dfrac{\chi_{s}(x)}{|B_{s}|}\int_{W_{s}}\sum_{k\succeq s}f_k,
\end{eqnarray}
where the second sum is indexed over all the $s\in\Gamma$ such that $t$ is the parent of $s$. In the particular case when $t$ is the root of $\Gamma$, the formula (\ref{tree decomposition}) means 
\[\tilde{f}_a(x)=-\sum_{s:\,s_p=a}\dfrac{\chi_{s}(x)}{|B_{s}|}\int_{W_{s}}\sum_{k\succeq s}f_k.\]

In the next theorem we prove that $f$ can be written as $\sum_{t\in\Gamma} f_t-\tilde{f}_t$, and some other properties of the decomposition, but before that let us define an important operator and prove its continuity. Let $T:L^p(\O)\to L^p(\O)$ be the operator defined by
\begin{eqnarray}\label{Ttree}
Tf(x):=\sum_{a\neq t\in\Gamma}\dfrac{\chi_{t}(x)}{|W_t|}\int_{W_t}|f|.
\end{eqnarray} 

\begin{lemma}\label{Ttreecont} The operator $T:L^p(\O)\to L^p(\O)$ defined in (\ref{Ttree}) is weak $(1,1)$ continuous and strong $(p,p)$ continuous for $1<p\leq \infty$ with
\[\|T\|_{L^p\to L^p} \leq 2\left(\dfrac{pN}{p-1}\right)^{1/p}.\]
\end{lemma}

\begin{proof} 
We prove first that $T$ is strong $(\infty,\infty)$ continuous and weak $(1,1)$ continuous. Then, using an interpolation theorem, we extend the result to all $1<p<\infty$. 

$T$ is an average of $f$ when it is not zero, thus by a straightforward calculation, it can be proved that $T$ is continuous  from $L^\infty$ to $L^\infty$, with norm $\|T\|_{L^\infty\to L^\infty} \leq 1$. In order to prove the weak $(1,1)$ continuity, and given $\lambda>0$, we define the subset of minimal vertices $\Gamma_0\subseteq \Gamma$ as 
\begin{eqnarray*}
\Gamma_0:=\left\{t\in \Gamma\,:\,\frac{1}{|W_t|}\int_{W_t}|f|>\lambda\text{ and }\frac{1}{|W_s|}\int_{W_s}|f|\leq\lambda\text{ for all }s\preceq t \text{ different from }t\right\}.
\end{eqnarray*}
Thus, 
\begin{eqnarray*}
|\{x\in\O\,:\,Tf(x)>\lambda\}|&\leq& \sum_{t\in \Gamma_0}|W_{t}|\\
&<& \frac{1}{\lambda} \sum_{t\in\Gamma_0}\int_{W_{t}}|f|\leq \dfrac{N}{\lambda} \|f\|_{L^1(\O)},
\end{eqnarray*}
where $N$ was defined in {\bf (a)} on page \pageref{a} and it controls the overlapping of the collection $\{\O_t\}$.
Thus, $T$ is weak $(1,1)$ continuous with norm lesser than or equal to $N$.

Finally, using Marcinkiewicz interpolation (see Theorem 2.4 on \cite{D}) $T$ is strong $(p,p)$ continuous, and its norm is lesser than $2\left(\dfrac{pN}{p-1}\right)^{1/p}$.
\end{proof}

Now, we define the weight $\o:\O\to \R_+$ by 
\begin{equation}\label{TreeWeight}
\o(x):=\left\{
  \begin{array}{l l}
     \dfrac{|B_t|}{|W_t|} & \quad \text{if $x\in B_t$ for some $t\in\Gamma$, $t\neq a$}\\
     \\
     1 & \quad \text{otherwise}.\\
   \end{array} \right.
\end{equation}

Let us observe that the collection $\{B_t\}_{t\in\Gamma}$ is pairwise disjoint, thus the weight is well defined. Moreover, $0<\omega(x)\leq 1$ for all $x\in\Omega$.

\begin{theorem}[Decomposition technique]\label{tree} Let $\O\subset\R^n$ be a bounded domain for which there exists a decomposition $\{\O_t\}_{t\in\Gamma}$ that fulfills  {\bf (a)} and {\bf (b)}. Given $f\in L^1(\O)$, and $1<p<\infty$, the decomposition $f=\sum_{t\in\Gamma}f_t-\tilde{f}_t$ defined on (\ref{tree decomposition}) satisfies that $Supp(f_t-\tilde{f}_t)\subset \O_t$, $\int_{\Omega_t} f_t-\tilde{f}_t=0$ for all $t\neq a$,  $\int_{\Omega_a} f_a-\tilde{f}_a=\int_\O f$, and
\begin{equation}\label{normtree}
\sum_{t\in\Gamma}\|f_t-\tilde{f}_t\|^p_{L^p(\O_t,\o)} \leq C_1 \|f\|^p_{L^p(\O)},
\end{equation}
where $C_1=2^pN\left(1+\dfrac{2^{p+1}p}{p-1}\right)$.

In addition, if $\ho_1,\ho_2:\O\to\R_{>0}$ are two weights satisfying that $L^p(\O,\ho_2)\subset L^1(\O)$, and  the identity operator $I$ and $T$ are continuous from $L^p(\O,\ho_2)$ to $L^p(\O,\ho_1)$ with norms $M_I$ and $M_T$, 
the mentioned decomposition also satisfies the following estimate
\begin{equation}\label{normtree2}
\sum_{t\in\Gamma}\|f_t-\tilde{f}_t\|^p_{L^p(\O_i,\o\ho_1)} \leq C_2 \|f\|^p_{L^p(\O,\ho_2)},
\end{equation}
where $C_2=2^p(NM^p_I+2M^p_T)$.
\end{theorem}

\begin{proof} Observe that $B_t$ and all the $B_s$ on the identity (\ref{tree decomposition}) are included in $\O_t$, thus it follows that  $Supp(f_t-\tilde{f}_t)\subset \O_t$.  

Let us remark that $\tilde{f}_t(x)\neq 0$ only if $x$ belongs to $B_s$ for some $s\in\Gamma-\{a\}$ with $t=s$ or $t=s_p$. Moreover, given $x$ in $B_s$ it follows that $\tilde{f}_s(x)+\tilde{f}_{s_p}(x)=0$, concluding that 
\[\sum_{t\in \Gamma}f_t(x)-\tilde{f}_t(x)=\sum_{t\in \Gamma}f_t(x)-\sum_{t\in \Gamma}\tilde{f}_t(x)=f(x)+0.\]

On the other hand, in order to prove the vanishing mean value of $f_t-\tilde{f}_t$, with $t\neq a$,  
\begin{eqnarray*}
\int_{\O_t}\tilde{f}_t=\int_{W_t}\sum_{k\succeq t}f_k-\sum_{s:\,s_p=t}\int_{W_{s}}\sum_{k\succeq s}f_k=\sum_{k\succeq t}\int_{\O}f_k-\sum_{\substack{k\succeq t\\ k\neq t}}\int_{\O}f_k=
\int_{\O_t}f_t,\end{eqnarray*}
obtaining that $\int_{\O_t}f_t-\tilde{f}_t=0$.  The case $t=a$ follows from
\begin{eqnarray*}
\int_{\O_a}f_a-\tilde{f}_a=\int_{\O_a}f_a+\sum_{s:\,s_p=a}\int_{W_{s}}\sum_{k\succeq s}f_k=\int_{\O}f_a+\sum_{\substack{k\neq a}}\int_{\O}f_k=
\int_{\O}f.\end{eqnarray*}

Let us continue with the proof of (\ref{normtree2}). Thus,

\begin{eqnarray*}
\sum_{t\in\Gamma}\|f_t-\tilde{f}_t\|^p_{L^p(\O_t,\o\ho_1)} &\leq& 2^p \sum_{t\in\Gamma}\|f_t\|^p_{L^p(\O_t,\o\ho_1)}+2^p\sum_{t\in\Gamma}\|\tilde{f}_t\|^p_{L^p(\O_t,\o\ho_1)}=(I')+(II').
\end{eqnarray*}

Using that $0\leq \phi_i,\,\o\leq 1$ and the overlapping of the collection $\O_t$ is not greater than $N$, it follows that 
\[(I')\leq 2^p N\|f\|^p_{L^p(\O,\ho_1)}\leq 2^p N M^p_{I} \|f\|^p_{L^p(\O,\ho_2)}.\]

On the other hand, using that the collection $\{B_t\}_{t\neq a}$ is pairwise disjoint, it can be observed for any $t\neq a$ that

\begin{eqnarray*}
\left(|\tilde{f}_t(x)|\o(x)\ho_1(x)\right)^p&\leq&\left(\dfrac{\o(x)\chi_t(x)}{|B_t|}\int_{W_t}|f|\ +\ \sum_{s:\,s_p=t}\dfrac{\o(x)\chi_s(x)}{|B_s|}\int_{W_s}|f|\right)^p\ho_1(x)^p\\
&=&\left(\dfrac{\chi_t(x)}{|W_t|}\int_{W_t}|f|\ +\ \sum_{s:\,s_p=t}\dfrac{\chi_s(x)}{|W_s|}\int_{W_s}|f|\right)^p\ho_1(x)^p\\
&=&\left(\left(\dfrac{\chi_t(x)}{|W_t|}\int_{W_t}|f|\right)^p\ +\ \sum_{s:\,s_p=t}\left(\dfrac{\chi_s(x)}{|W_s|}\int_{W_s}|f|\right)^p\right)\ho_1(x)^p.
\end{eqnarray*}

The case $t=a$ is analogous. Hence, 
\begin{eqnarray*}
\sum_{t\in\Gamma}\int_{\O_t}\left(|\tilde{f}_t(x)|\o(x)\ho_1(x)\right)^p &\leq& 2\int_\O \sum_{a\neq t\in\Gamma}\left(\dfrac{\chi_t(x)}{|W_t|}\int_{W_t}|f|\right)^p\ho_1(x)^p\\
&=& 2\int_\O \left(\sum_{a\neq t\in\Gamma}\dfrac{\chi_t(x)}{|W_t|}\int_{W_t}|f|\right)^p\ho_1(x)^p.
\end{eqnarray*}

Finally,
\begin{eqnarray*}
(II')&\leq& 2^{p+1} \int_\O (Tf(x)\ho_1(x))^p \leq 2^{p+1}M^p_T\|f\|^p_{L^p(\O,\ho_2)},
\end{eqnarray*}
ending the proof of (\ref{normtree2}). 

Using the continuity of $T$ proved in Lemma \ref{Ttreecont}, it can be seen that (\ref{normtree}) follows from (\ref{normtree2}) 
\end{proof}

\subsection{An application: Divergence problem}

In this subsection, we apply Theorem \ref{tree} to show the existence of a weighted solution of the divergence problem on some bounded domains $\O\subset\R^n$. In fact, this result can be applied if there exists a collection $\{\O_t\}_{t\in\Gamma}$ of subdomains of $\O$ verifying {\bf (a)} and {\bf (b)} in Subsection \ref{def tree}, and the additional four conditions stated bellow:
\begin{enumerate}
\item[{\bf (c)}] \label{c}For any point $x\in \Omega$ there exists an open set $U$ containing $x$ such that $U\cap \Omega_t\neq \emptyset$ for a finite number of $\O_t$'s (this finite number does not need to be bounded by $N$). 

\item[{\bf (d)}] \label{d}There exists a weight $\bo:\O\to\R_{>0}$ such that
\begin{equation*}
\esssup_{x\in \O_t}\bo(x)\leq M_1 \essinf_{x\in \O_t} \o(x),
\end{equation*}  
for all $t\in\Gamma$, where $\omega$ is the weight defined in (\ref{TreeWeight}) and $M_1$ is independent of $t$.
\end{enumerate}

In the next two conditions $\ho:\O\to \R_{>0}$ is a weight such that $L^p(\O,\ho)\subset L^1(\O)$, with $1<p<\infty$.
\begin{enumerate}
\item[{\bf (e)}] \label{e} Given $g\in L^p(\O_t,\ho)$, with vanishing mean value, there exists a solution $\vv\in W_0^p(\O_t,\ho)^n$ of $\di\vv=g$ with   
\begin{eqnarray*}
\|D\vv\|_{L^p(\O_t,\ho)}\leq M_2 \|g\|_{L^p(\O_t,\ho)},
\end{eqnarray*}
for all $t\in\Gamma$, where the positive constant $M_2$ does not depend on $t$. 
\item[{\bf (f)}] \label{f} The operator $T$ defined in (\ref{Ttree}) is continuous from $L^p(\Omega,\ho)$ to itself with norm $M_T$.
\end{enumerate}

An example of a collection of subdomains verifying {\bf (e)} could be $\{\O_t\}_{t\in\Gamma}$ such that the constant $C_{\O_t}$ is uniformly bounded (for example, cubes) and the weight $\ho$ satisfies that 
\begin{equation*}
\esssup_{x\in \O_t}\ho(x)\leq C \essinf_{x\in \O_t} \ho(x),
\end{equation*}  
where $C$ is independent of $t$. 

The condition {\bf (f)} is used to include the weight $\ho$ in both sides of the inequality (\ref{Div estimate}). The case when $\ho=1$ was proved in general in Lemma \ref{Ttreecont}.

\begin{theorem}\label{Div} Let $\Omega\subset\R^n$ be a bounded domain, $\ho,\bo:\O\to \R_{>0}$ two weights, with $L^p(\O,\ho)\subset L^1(\O)$, for $1<p<\infty$, and finally a collection $\{\O_t\}_{t\in\Gamma}$ of subdomains of $\O$ verifying the conditions from {\bf (a)} to {\bf (f)} mentioned above. Hence, given $f\in L^p_0(\O,\ho)$ with vanishing mean value there exists a solution $\uu\in W_0^{1,p}(\O,\bo\ho)^n$ of $\di\uu=f$ such that
\begin{eqnarray}\label{Div estimate}
\|D\uu\|_{L^p(\Omega,\bo\ho)}\leq C \|f\|_{L^p(\O,\ho)},
\end{eqnarray}
where
 \[C= 2NM_1M_2\left(N+2M_T^p\right)^{1/p}.\]
\end{theorem}
\begin{proof} The collection of subdomains satisfies {\bf (a)} and {\bf (b)}, and from {\bf (f)} the weight $\ho$ makes the operator $T:L^p(\O,\ho)\to L^p(\O,\ho)$ continuous. Thus, using Theorem \ref{tree} we can decompose 
the integrable function $f$ as \[f=\sum_{t\in\Gamma}f_t-\tilde{f}_t,\]
where $f_t-\tilde{f}_t\in L^p(\O_t,\ho)$, with vanishing mean value, and 
\begin{equation*}
\sum_{t\in\Gamma}\left(\essinf_{x\in \O_t} \o(x)\right)^p  \|f_t-\tilde{f}_t\|^p_{L^p(\O_t,\ho)} \leq  \sum_{t\in\Gamma}\|f_t-\tilde{f}_t\|^p_{L^p(\O_t,\o\ho)} \leq C_2\|f\|^p_{L^p(\O,\ho)},
\end{equation*}
where $C_2=2^p\left(N+2M_T^p\right)$.

Note that the essential infimum of $\o$ over $\O_t$ is positive because of {\bf (d)}, then $f_t-\tilde{f}_t$ belongs to $L^p(\O_t,\ho)$ as we announced. Now, using condition {\bf (e)} there exists a solution $\uu_t\in W_0^p(\O_t,\ho)^n$ of $\di\uu_t=f_t-\tilde{f}_t$, with \[\|D\uu_t\|_{L^p(\O_t,\ho)}\leq M_2 \|f_t-\tilde{f}_t\|_{L^p(\O_t,\ho)},\]
where $M_2$ is independent of $t$.
Therefore, using requirement {\bf (c)}, the vector field $\uu:=\sum_{t\in\Gamma}\uu_t$ is a solution of $\di\uu=f$. Moreover, using {\bf (d)}
\begin{eqnarray*}
\|D\uu\|^p_{L^p(\O,\bo\ho)}&\leq& N^p \sum_{t\in\Gamma}\|D\uu_t\|^p_{L^p(\O_t,\bo\ho)}\leq N^p M_1^p \sum_{t\in\Gamma}\left(\essinf_{x\in \O_t}\o(x)\right)^p\|D\uu_t\|^p_{L^p(\O_t,\ho)}\\
&\leq& N^p M_1^p M_2^p\sum_{t\in\Gamma}\left(\essinf_{x\in \O_t}\o(x)\right)^p\|f_t-\tilde{f}_t\|^p_{L^p(\O_t,\ho)}\\
&\leq& N^p M_1^pM_2^p 2^p\left(N+2M_T^p\right)\|f\|^p_{L^p(\O,\ho)},
\end{eqnarray*}
proving that $\uu$ belongs to $W^{1,p}(\Omega,\bo\ho)^n$ and the estimate claimed in the theorem.

Finally, let us prove that $\uu$ belongs to $\overline{C_0^\infty(\O)^n}$. Given $\epsilon>0$, using that $\sum_{t\in\Gamma}\|D\uu_t\|^p_{L^p(\O_t,\bo\ho)}<\infty$, there exists a finite set $\Gamma_0\subset\Gamma$ such that 
\[\sum_{t\in\Gamma\setminus\Gamma_0}\|D\uu_t\|^p_{L^p(\O_t,\bo\ho)}<N^{-p}\frac{\epsilon}{2}.\]
Now, for $t\in\Gamma_0$, we take $\vv_t\in C_0^\infty(\Omega_t)^n$ such that
\[\|D\uu_t-D\vv_t\|^p_{L^p(\O_t,\ho)}\leq N^{-p} M_1^{-p} \left(\essinf_{x\in \O_t}\o(x)\right)^{-p} \frac{\epsilon}{2m},\]
where $m$ is the cardinal of $\Gamma_0$.
Thus, using that each $\O_t$ is included in $\O$ and $\Gamma_0$ is finite, $\vv:=\sum_{t\in\Gamma_0}\vv_t$ belongs to $C_0^\infty(\O)^n$ and 
\begin{eqnarray*}
& &\|D\uu-D\vv\|^p_{L^p(\O,\bo\ho)}\\
&\leq&N^p \sum_{t\in\Gamma\setminus\Gamma_0}\|D\uu_t\|^p_{L^p(\O_t,\bo\ho)}   +   N^p M_1^p \sum_{t\in\Gamma_0}\left(\essinf_{x\in \O_t}\o(x)\right)^p\|D\uu_t-D\vv_t\|^p_{L^p(\O_t,\ho)}< \epsilon,
\end{eqnarray*}
completing the proof.
\end{proof}

In the next corollary we prove that $\O$ satisfies $\divp$, with an estimate over the constant $C_\O$, if it is possible to decompose $\O$ by a good enough collection of subdomains $\{\O_t\}_{t\in\Gamma}$.

\begin{corollary}\label{Rooms and Corridors} Let $\O\subset\R^n$ be a bounded domain for which there exists a decomposition $\{\O_t\}_{t\in\Gamma}$ that fulfills  {\bf (a)}, {\bf (b)}, {\bf (c)} and {\bf (e)} for $\ho=1$ and $1<p<\infty$ such that $\omega(x)\geq \frac{1}{M_1}$ for all $x \in \Omega$. Hence,  
$\O$ satisfies $\divp$ and its constant $C_\O$ is bounded by \[C_\O\leq 2M_1M_2N^{1+1/p}\left(1+\frac{2^{p+1}p}{p-1}\right)^{1/p}.\]
\end{corollary}
\begin{proof}
This result is a consequence of the previous theorem using $\ho=\bo=1$ and Lemma \ref{Ttreecont}. 
\end{proof}

\section{Divergence problem on general domains}\label{General domains}

In this section, we show the existence of a solution in a weighted Sobolev space of the divergence problem, $\di\uu=f$, on an arbitrary bounded domain $\O\subset\R^n$. The constant involved in the estimation of the solution is explicit and depends only on $n$ and $p$. Furthermore, we use Whitney cubes to decompose the domain $\O$, and the weight $\o$ that we obtain for this decomposition depends locally on the ratio $\frac{|Q|}{|S(Q)|}$, where $Q$ is a Whitney cube and $S(Q)$ is its shadow. Thus, it could be of interest to consider domains where this ratio is studied. 

In \cite{DMRT}, the authors prove a similar result also for arbitrary bounded domains using an atomic decomposition obtained from a weighted Poincar\'e inequality, where the weight is related to the Euclidean geodesic distance in $\O$.

Let $\mathcal{W}:=\{Q_t\}_{t\in\Gamma}$ be a Whitney decomposition, i.e. a family of closed dyadic cubes whose interiors are pairwise disjoints, which satisfies 
\begin{enumerate}
\item[(i)] $\Omega=\bigcup_{t\in\Gamma}Q_t$, 
\item[(ii)] $diam(Q_t)\leq dist(Q_t,\partial\Omega)\leq 4\,diam(Q_t)$,
\item[(iii)] $\frac14\,diam(Q_s)\leq diam (Q_t) \leq 4\,diam(Q_s)$, when $Q_s\cap Q_t\neq \emptyset$,
\end{enumerate}
where $diam(Q)$ denotes the diameter of $Q$. Moreover, given a constant $\varepsilon\in (0,1/4)$ which is arbitrary but will be kept fixed in what follows, $Q_t^*$ denotes the open cube which has the same center as $Q_t$ but is expanded by the factor $1+\varepsilon$. This collection of expanded cubes satisfies that  
\begin{eqnarray}\label{Whitney overlapping}
\sum_{t\in\Gamma}\chi_{Q^*_t}(x)\leq 12^n \chi_{\Omega}, \text{ for all } x\in\R^n.
\end{eqnarray}
Moreover, $Q^*_s$ intersects $Q_t$ if and only if $Q_s$ touches $Q_t$. See \cite{St} for details. 

Now, let us take a Whitney cube $Q_a$ which will be distinguished from the rest. Then, for each $Q_t$ in $\mathcal{W}$ we take a unique chain of cubes $\Omega_a=\Omega_{t_0},\Omega_{t_1},\cdots,\Omega_{t_k}=\Omega_t$ connecting $Q_a$ with $Q_t$, such that for each $1\leq i\leq k$ the intersection between $Q_{t_i}$ and $Q_{t_{i-1}}$ is a $n-1$ dimensional face of one of those cubes. In addition, we assume that $k$ is minimal over this type of chains and, using an inductive argument, that $\Omega_a,\Omega_{t_1},\cdots,\Omega_{t_i}$ is the chain taken for each $\Omega_{t_i}$, with $1\leq i\leq k$.

Observe that using these chains connecting any $Q_t\in \mathcal{W}$ with $Q_a$ in a unique way it is possible to define a rooted tree structure over $\Gamma$. Indeed, we say that two vertices $s,s'\in \Gamma$ are connected by an edge if and only if $Q_s$ and $Q_{s'}$ are consecutive cubes in a chain $Q_a,Q_{t_1},\cdots,Q_{t_k}=Q_t$, for some $Q_t$. As it is expected, the root of $\Gamma$ is the index $a$ of the distinguished cube $Q_a$. In addition, a partial order $\preceq$ over $\Gamma$ is inherited.

Now, we define the shadow $S(Q_t)$ of a cube $Q_t\in \mathcal{W}$ as 
\begin{equation}\label{shadow}
S(Q_t):=\bigcup_{s\succeq t} Q^*_s.
\end{equation}

\begin{theorem}\label{Divergence General domain} Let $\O\subset\R^n$ be an arbitrary bounded domain, and $1< p<\infty$.  Given $f\in L^p_0(\Omega)$ there exists a vector field $\uu\in W^{1,p}_0(\Omega,\bo)^n$  solution of $\di\uu=f$ with the estimate 
\begin{eqnarray}\label{Div general domain}
\|D\uu\|_{L^p(\Omega,\bo)}\leq C_{p,n} \|f\|_{L^p(\O)},
\end{eqnarray}
where $C_{p,n}$ depends only on $n$ and $p$, and $\bo$ is defined over the interior of $Q_s$ as 
\begin{eqnarray*}
\bo:=\min_{Q_k\cap Q_s\neq \emptyset}\dfrac{|Q_k|}{|S(Q_k)|}.
\end{eqnarray*}
The expanded cubes $Q^*_s$ use in their definition $\varepsilon:=2^{-7}$.
\end{theorem}
\begin{proof} Let us observe that $\bo$ is defined almost everywhere but it is sufficient in this context. 

This result is a consequence of Theorem \ref{Div}. Let us consider the decomposition $\{\O_t\}_{t\in\Gamma}$ defined by $\O_t:=Q^*_t$ with $\varepsilon:=2^{-7}$.
From (\ref{Whitney overlapping}), it follows that the overlapping of the subdomains of this collection is bounded by $N=12^n$. 

Now, let us define the collection $\{B_t\}_{a\neq t\in\Gamma}$. Given $t$ in $\Gamma\setminus\{a\}$ and its parent $t_p$, the intersection $Q_t\cap Q_{t_p}$ is a $n-1$ dimensional face of one of those. Let us denote by $c^t$ the center of that $n-1$ dimensional cube and observe that the length of its sides is greater than $l_t/4$, where $l_t$ is the length of the sides of $Q_t$. Thus, if we define the distance $d_\infty(x,y):=\max_{1\leq i\leq n}|x_i-y_i|$ it follows that $d_\infty(c_t,Q_s)\geq l_t/8$ for any $s\neq t,t_p$.
Now, we define $B_t$ as an open cube with center in $c^t$ and sides with length $l$ small enough in order to have the cube $B_t$ included in $Q^*_t\cap Q^*_{t_p}$ and disjoint from $Q^*_s$ for all $s\neq t, t_p$. It is known that $l_{t_p}\geq l_t/4$, then taking $l=\varepsilon l_t/4$ it follows that $B_t\subset Q^*_{t}\cap Q^*_{t_p}$. In particular, as $\varepsilon<1$ it can be seen that  $B_t\subset Q_t\cup Q_{t_p}$. Hence, using that $Q_s^*$ intersects $Q_k$ if and only if $Q_s$ touches $Q_k$, we can assert that if $Q^*_s$ intersects $Q_t\cup Q_{t_p}$ then the common length of the sides of $Q_s$ is lesser than $16l_t$. Thus, $d_\infty(Q^*_s,c^t)\geq l_t/8-16l_t\varepsilon/2$. Thus, it is sufficient to show that $\varepsilon$ verifies that  $l_t/8-16l_t\varepsilon/2> \varepsilon l_t/8$. Hence, $B_t\subset \O_t\cap \O_{t_p}$ and $B_t\cap \O_s=\emptyset$ if  $s\neq t,t_p$, obtaining a collection $\{B_t\}_{t\neq a}$ pairwise disjoint. 
Thus, the collection $\{\O_t\}_{t\in\Gamma}$ of subdomains of $\O$ verifies conditions {\bf (a)} and {\bf (b)} in Subsection \ref{def tree}. 

In order to prove condition {\bf (c)} on page \pageref{c}, we can see that for any $x\in \O$ the ball with center $x$ and radius $\frac{1}{2}d(x,\partial\O)$ intersects only a finite number of $Q_s^*$'s. The condition {\bf (e)} is obtained by observing that $\ho=1$ and the subdomains in the decomposition are cubes thus the constants $C_{\O_t}$ are equal to each other (see Lemma \ref{rectangles}). Finally, condition {\bf (f)} has been proved in Lemma \ref{Ttreecont}, thus it only remains to prove {\bf (d)}.  

Now, given $t\in \Gamma$ , it can be observed that $\o(x)\neq 1$ over $Q_t$ only if $x$ belongs to $B_s$ with $s=t$ or $s_p=t$. Thus, given $x\in Q_t$ it follows that
\begin{eqnarray*}
\o(x)=\left\{
  \begin{array}{l l}
     \dfrac{|B_t|}{|S(Q_t)|}=\dfrac{2^{-9n}|Q_t|}{|S(Q_t)|}& \quad \text{if $x\in B_t$}\\
     \\
     \dfrac{|B_s|}{|S(Q_s)|}=\dfrac{2^{-9n}|Q_s|}{|S(Q_s)|}\geq \dfrac{2^{-11n}|Q_t|}{|S(Q_t)|} & \quad \text{if $x\in B_s$ with $s_p= t$}\\
     \\
     \hspace{0.6cm} 1& \quad \text{otherwise.}
   \end{array} \right.
\end{eqnarray*}
Hence, using that  $\O_t$ is included in $\displaystyle{\bigcup_{Q_s\cap Q_t\neq \emptyset} Q_s}$, it follows that  
\[\esssup_{x\in \O_t}\bo(x)=\max_{Q_s\cap Q_t\neq \emptyset} \esssup_{x\in Q_s} \bo(x) \leq \dfrac{|Q_t|}{|S(Q_t)|} \leq  2^{11n} \essinf_{x\in \O_t} \o(x),\]  
obtaining the condition {\bf (e)} with $M_1=2^{11n}$. Thus, using Theorem \ref{Div} we obtain (\ref{Div general domain}) with \[C_{p,n}= 2C_Q2^{11n}12^{n+n/p}\left(1+\frac{2^{p+1}p}{p-1}\right)^{1/p},\]
where $C_Q$ is the constant in (\ref{Introduction div}) for an arbitrary cube $Q$.

\end{proof}

\section{Divergence problem and Stokes equations on domains with an external cusp}\label{Cuspidal domains}

In this section we show the second application of our decomposition to prove the existence of a weighted solution of $\di\uu=f$ on a class of $n$ dimensional domains with an external cusp arbitrarily narrow. 
Similar domains were studied in \cite{DL2} where the cusp is defined by a power function $x^\gamma$, with $\gamma>1$.

Given a Lipschitz function $\varphi:[0,a]\to \R$ that satisfies the properties:  
\begin{enumerate}
\item[(i)] $\vfi(0)=0$, and $\vfi(r)>0$ if $x\in (0,a]$,
\item[(ii)] $\vfi'(0)=0$, $|\vfi'|\leq K_1$,
\item[(iii)]  $\frac{\vfi(t)}{t}\leq K_2\frac{\vfi(r)}{r}$, for all $0<t<r\leq a$,
\end{enumerate}
we define the following domain with a cusp at the origin:
\begin{equation}\label{domain}
\O_\varphi:=\{(x,y)\in\R\times\R^{n-1}\,:\, 0<x<a\text{ and }|y|< \varphi(x)\}\subset\R^n.
\end{equation}

\bigskip

The following are three examples of functions which verify (i), (ii) and (iii):
\begin{itemize}
\item $\vfi(x)=x^\gamma$, with $\gamma>1$.
\item $\vfi(x)=e^{-1/x^2}$ in $(0,a]$ and $\vfi(0)=0$.
\item $\vfi(x)=x^\gamma (2+\sin(x^{1-\gamma}))$ in $(0,a]$, with $\gamma>1$, and $\vfi(0)=0$.
\end{itemize}

\bigskip

The Lipschitz condition in (ii) keeps $\partial\Omega_\vfi$ from having cusps different from the one at the origin, and the condition (iii) prevents the domain to be of type of ``Rooms and corridors'', where the weight could be worse than the one considered by us. Moreover, it is used to solve a technical issue when $\kappa$ in Theorem \ref{Divergence in cuspidal} is positive.

Let us start introducing a decreasing sequence in the interval $(0,a]$ and some of its properties. This sequence will be used to define a decomposition of $\O_\vfi$. Thus, we define inductively a decreasing sequence $\{x_i\}_{i\geq 0}$ in $(0,a]$ with $x_0=a$, and $x_{i+1}$ the maximum number in $(0,x_i)$ satisfying that $\vfi(x)=x_i-x$. The well definition of this sequence is based on the continuity of $\vfi(x)+x$, which satisfies that $\vfi(0)+0<x_i$ and $\vfi(x_i)+x_i>x_i$, using that $\vfi$ is continuous and positive on $(0,a]$, with $\vfi(0)=0$. In addition, it can be seen that $\{x_i\}_{i\geq 0}$ decreases to 0. Indeed, if $\{x_i\}_{i\geq 0}$ converges to $\bar{x}\geq 0$, then \[\vfi(\bar{x})=\lim_{i\to \infty}\vfi(x_{i+1})=\lim_{i\to \infty} x_i-x_{i+1}=0.\] 
Hence, $\bar{x}=0$.

\begin{figure}[htb]
       \center{\includegraphics[width=80mm]{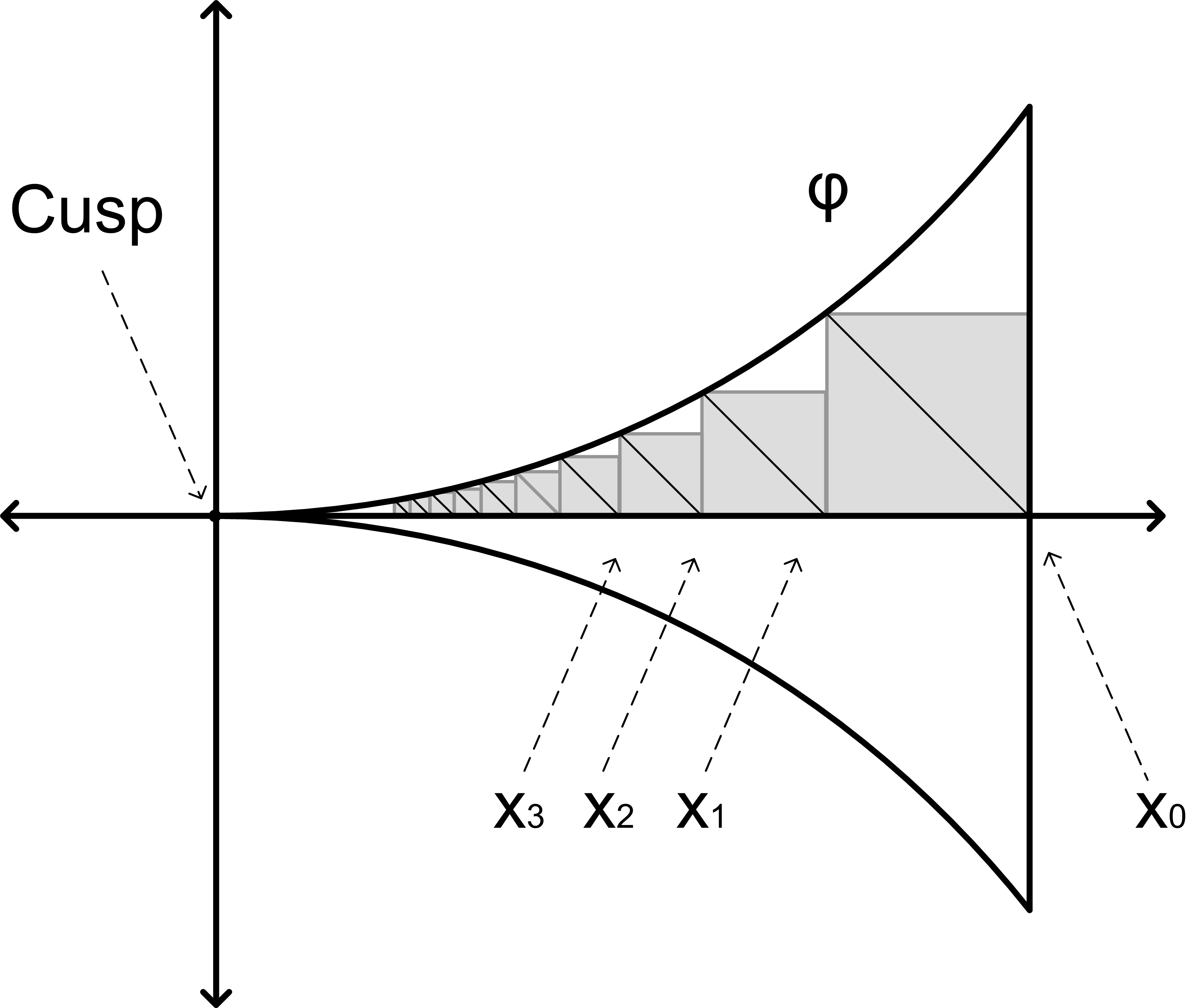}}
        \caption{ A simple example with an increasing $\vfi$}
\end{figure}

Let us see some properties of this sequence. Taking $x_{i+1}\leq x\leq x_i$, we can assert that 
\begin{equation}\label{r}
x_{i+1}\leq x\leq (K_1+1)x_{i+1}
\end{equation} 
and 
\begin{equation}\label{varphi}
\frac{1}{K_2}\vfi(x_{i+1})\leq\vfi(x)\leq (K_1+1)\vfi(x_{i+1}).
\end{equation} 

The first inequality in (\ref{r}) holds by definition and the second one can be proved 
observing that \[x\leq x_i=\vfi(x_{i+1})+x_{i+1}=\dfrac{\vfi(x_{i+1})-\vfi(0)}{x_{i+1}-0}x_{i+1}+x_{i+1}.\]

The second inequality in (\ref{varphi}) follows from
\[|\vfi(x)-\vfi(x_{i+1})|\leq K_1(x-x_{i+1})\leq K_1(x_i-x_{i+1})=K_1\vfi(x_{i+1}),\]
and the first one can be proved by 
\[\vfi(x_{i+1})\leq K_2\vfi(x)\,\dfrac{x_{i+1}}{x}\leq K_2\vfi(x).\]

\bigskip

Now, we introduce the collection $\{\O_i\}_{i\in\N_0}$ of subdomains of $\O_\vfi$ to be used with Theorem \ref{Div}. Indeed, given $i\in \N_0$ we define
\begin{eqnarray}\label{Omegai}
\O_i:=\{(x,y)\in \O_\vfi\,:\, x_{i+2}<x<x_{i}\}.
\end{eqnarray}
Note that in this case $\Gamma=\N_0$ where two vertices $i$ and $j$ are connected by an edge if and only if $|i-j|=1$. Moreover, if we take the root $a=0$, the partial order $\preceq$ inherited from this tree structure coincides with the total order $\leq$ of $\N_0$.

\bigskip

The following theorem is the main result of this section.

\begin{theorem}[Divergence on cuspidal domains]\label{Divergence in cuspidal} Let $\O_\varphi\subset\R^n$ be the domain defined in (\ref{domain}), $1<p<\infty$ and $\kappa\geq0$. Given $f\in L^p(\O,\vp^{-\kappa})$, with vanishing mean value, there exists a solution $\uu$ in $W_0^{1,p}(\O_\vfi,\vp^{1-\kappa})^n$ of $\di \uu= f$ satisfying that 
\begin{equation}\label{normdiv}
\|D\uu\|_{L^p(\O_\vfi,\vp^{1-\kappa})}\leq C \|f\|_{L^p(\O_\vfi,\vp^{-\kappa})},
\end{equation}
where $\vp(x,y)=\dfrac{\varphi(x)}{x}$, and $C$ depends only on $K_1$, $K_2$, $p$, $n$ and $\kappa$.
\end{theorem}

\begin{proof} As we mentioned before this theorem is an application of Theorem \ref{Div}. The collection of subdomains $\{\O_i\}_{i\in \N_0}$ has been defined in (\ref{Omegai}), and the weights $\ho,\bo:\O_\vfi\to\R_{>0}$ are 
defined by $\bo:=\vp$ and $\ho:=\vp^{-\kappa}$. Observe that $\Omega_i$ is indeed a subdomain of $\O_\vfi$. Moreover, $\vp\leq K_1$ and $\kappa\geq 0$, thus $L^p(\O_\vfi,\vp^{-\kappa})\subset L^1(\O_\vfi)$. Then, it just remains to prove the conditions {\bf (a)} to {\bf (f)} on pages \pageref{a} and \pageref{c}.

Just from the definition of the collections of subdomains it can be observed that conditions {\bf (a)} and {\bf (c)} hold, with a constant $N=2$. 
The collection $\{B_i\}_{i\geq 1}$ defined below verifies {\bf (b)}:
\begin{eqnarray*}
B_i:=\O_i\cap\O_{i-1}=\{(x,y)\in \O_\vfi\,:\, x_{i+1}<x<x_i\}. 
\end{eqnarray*}

Now, let us prove {\bf (d)}. The weight $\o$ defined in (\ref{TreeWeight}) is equal to $\o(x,y)=\frac{|B_i|}{|W_i|}$ over $B_i$, where  
\begin{eqnarray*}
W_i&=&\bigcup_{k\geq i}\O_k=\{(x,y)\in \O_\vfi\,:\, x<x_i\}. 
\end{eqnarray*}
Thus, given $(x,y)\in B_i$, for $i\geq 1$, using inequalities (\ref{r}) and (\ref{varphi}), and (iii), it follows that 
\begin{eqnarray*}
\o(x,y)=\dfrac{|B_i|}{|W_i|}\geq \dfrac{C_{K_2,n}\,\vfi(x_{i+1})^n}{C_{K_2,n}\, x_i \vfi(x_i)^{n-1}}
\geq C_{K_1,K_2,n}\dfrac{\vfi(x_i)}{x_i}.
\end{eqnarray*}
Now, given $(x,y) \in B_{i+1}$, and using that $\frac{\vfi(x_{i+1})}{x_{i+1}}\geq \frac{1}{K_1+1}\frac{\vfi(x_i)}{x_i}$ obtained from (\ref{r}) and (\ref{varphi}), we can conclude that 
\begin{eqnarray*}
\o(x,y)=\dfrac{|B_{i+1}|}{|W_{i+1}|}\geq C_{K_1,K_2,n}\dfrac{\vfi(x_{i+1})}{x_{i+1}}\geq C_{K_1,K_2,n}\dfrac{\vfi(x_i)}{x_i}.
\end{eqnarray*}
Hence, using (iii) and the previous inequalities, 
\begin{equation*}
\esssup_{x\in \O_i}\vp(x)\leq K_2 \dfrac{\vfi(x_i)}{x_i}\leq C_{K_1,K_2,n} \essinf_{x\in \O_i} \o(x).
\end{equation*}

Now, let us prove {\bf (f)}, the continuity of the operator $T$ and an estimation of its norm. In order to simplify the notation we denote $\vp(x)$ instead of $\vp(x,y)$. The proof uses (iii), the fact that $\kappa\geq 0$, and the continuity of $T$ without weight shown in Lemma \ref{Ttreecont}. Indeed, 
\begin{eqnarray*}
\int_{\O_\vfi} |Tg(x,y)|^p \vp^{-p\k}(x)&=&\sum_{i\geq 1}\int_{B_i} \vp^{-p\k}(x)\left(\dfrac{1}{|W_i|}\int_{W_i}|g|\vp^{-\k}\vp^{\k}\right)^p\\
&\leq& \sum_{i\geq 1}\int_{B_i}  \vp^{-p\k}(x)K_2^{p\k}\vp^{p\k}(x_i)\left(\dfrac{1}{|W_i|}\int_{W_i}|g|\vp^{-\k}\right)^p\\
&\leq& K_2^{2p\k} \sum_{i\geq 1} \left(\dfrac{\vp(x_i)}{\vp(x_{i+1})}\right)^{p\k}\int_{B_i} \left(\dfrac{1}{|W_i|}\int_{W_i}|g|\vp^{-\k}\right)^p\\
&\leq& C \sum_{i\geq 1}\int_{B_i} \left(\dfrac{1}{|W_i|}\int_{W_i}|g|\vp^{-\k}\right)^p\\
&\leq& C \| g\vp^{-\k}\|^p_{L^p(\O_\vfi)}= C \|g\|^p_{L^p(\O_\vfi,\vp^{-\k})},
\end{eqnarray*}
where $C$ depends only on $K_1$, $K_2$, $p$ and $\k$.

Finally, it just remains to prove {\bf (e)}. Using (iii), (\ref{r}) and (\ref{varphi}), we obtain that 
\[\dfrac{1}{(K_1+1)^2}\dfrac{\vfi(x_i)}{x_i}\leq \dfrac{\vfi(x)}{x}\leq K_2\dfrac{\vfi(x_i)}{x_i},\]
for all $x_{i+2}\leq x\leq x_i$. Thus, we can assume that $\vp^{-\kappa}$ is the constant function with value $\left(\frac{\vfi(x_i)}{x_i}\right)^{-\kappa}$ over $\O_i$. 
Thus, it is enough to prove {\bf (e)} when $\ho=1$. The proof of this case is shown in Lemma \ref{Omegai constant}. 
\end{proof}

The next result follows immediately from Theorem \ref{Divergence in cuspidal}.

\begin{theorem}[Stokes on cuspidal domains]\label{Stokes in Cuspidal} Given $\O_\vfi\subset\R^n$ the domain defined in (\ref{domain}), $h\in L_0^2(\O_\vfi,\vp^{-1})$, and $\gg\in H^{-1}(\O_\vfi)^n$. There exists a unique solution $(\uu,p)\in H_0^1(\O_\vfi)^n\times L^2(\O_\vfi,\vp)$ of (\ref{Stokes}), with $\int_{\O_\vfi} p \vp^2=0$. Moreover, 
\begin{eqnarray*}
\|D\uu\|_{L^2(\O_\vfi)}+\|p\|_{L^2(\O_\vfi,\vp)}\leq C \left(\|\gg\|_{H^{-1}(\O_\vfi)}+\|h\|_{L^2(\O_\vfi,\vp^{-1})}\right),
\end{eqnarray*} 
where $\vp(x,y)=\dfrac{\varphi(x)}{x}$, and $C$ depends only on $\O_\vfi$.
\end{theorem}
\begin{proof} By Theorem \ref{Divergence in cuspidal}, there exists $\tilde{\vv}\in H_0^1(\O_\vfi)^n$ satisfying that $\di\tilde{\vv}=h$ and the estimate (\ref{normdiv}). Then, $\Delta\tilde{\vv}\in H^{-1}(\O_\vfi)^n$. Now, using the Theorem stated on page \pageref{Theorem DL2}, and Theorem \ref{Divergence in cuspidal}, we can conclude that there exists a unique solution $(\vv,p)\in H_0^1(\O_\vfi)^n\times L^2(\O_\vfi,\vp)$ of 
\begin{align}
\begin{cases}
-\Delta\vv+\nabla p&=\gg+\Delta\tilde{\vv}  \quad \text{in $\O$}\\ 
\di\vv&=0  \quad \text{in $\O$},\\
\end{cases}
\end{align}
with $\int_{\O_\vfi}p\vp=0$, and 
\begin{eqnarray*}
\|D\vv\|_{L^2(\O_\vfi)}+\|p\|_{L^2(\O_\vfi,\vp)}&\leq& C \left(\|\gg\|_{H^{-1}(\O_\vfi)}+\|\Delta\tilde{\vv}\|_{H^{-1}(\O_\vfi)}\right)\\
&\leq& C \left(\|\gg\|_{H^{-1}(\O_\vfi)}+\|f\|_{L^2(\O_\vfi,\vp)}\right).
\end{eqnarray*}
Finally, $(\uu,p)=(\vv+\tilde{\vv},p)$ is the solution mentioned in the theorem.
\end{proof}

Next, we prove a lemma used in the proof of Theorem \ref{Divergence in cuspidal}.

\begin{lemma}\label{Omegai constant} Let $\O_i\subset \R^n$ be, with $i\geq 0$, the domain defined on (\ref{Omegai}), and $1<p<\infty$. Then, $\O_i$ satisfies $\divp$ and there exists a constant $C$ depending only on $K_1$, $K_2$, $n$, and $p$ such that \[C_{\O_i}\leq C\] for all $i\geq 0$, where $C_{\O_i}$ is the constant in (\ref{Introduction div}).
\end{lemma}
\begin{proof} $\O_i$ is a Lipschitz domain, and it is well known that a domain of this type satisfies $\divp$. What this lemma states is the existence of a bound for $C_{\O_i}$ independent of $i$.
The idea to prove this result consists in showing that each $\O_i$ can be written as the finite union of certain star-shaped domains with respect to a  ball (the number of domains in the union does not depend on $i$), for which there exists an estimate of their constant, which let us apply Corollary \ref{Rooms and Corridors}. 

Let us recall the definition of this class of domains. A domain $U$ is {\bf star-shaped} with respect to a ball $B$ if and only if any segment with an end-point in $U$ and the other one in $B$ is contained in $U$. This class of domains have the following estimate over the constant on the divergence problem (\ref{Introduction div}): if $R$ denotes the diameter of $U$ and $\rho$ the radius of the ball $B\subset U$, the constant $C_U$ is bounded by 
\begin{equation}\label{Galdi}
C_U\leq C_{n,p} \left(\frac{R}{\rho}\right)^{n+1}.
\end{equation} 
See Lemma III.3.1 in \cite{G} for details.

Let us show the way to split $\O_i$ into the finite union. Without loss of generality, we assume that $K_1,K_2\geq 1$. Thus, from (\ref{r}) and (\ref{varphi}) it follows that 
\begin{eqnarray}\label{Star1}
\vfi(x_{i+1})\leq|x_i-x_{i+2}|\leq 2K_2\vfi(x_{i+1}),
\end{eqnarray} 
and 
\begin{eqnarray}\label{Star2}
\frac{1}{2K_1K_2}\vfi(x_{i+1})\leq \vfi(x)\leq 2K_1K_2\vfi(x_{i+1}),
\end{eqnarray}
if $x$ belongs to $[x_{i+2},x_i]$. Next, we take a natural number $m$ such that $m > 16K_1^2K_2^2$ (this number $m$ could be arbitrarily big but it is fixed) and the $m+1$ equidistant points  $r_0<r_1<\cdots<r_m$, with $r_0=x_{i+2}$ and $r_m=x_{i}$. Thus, for $1\leq j\leq m$, it follows that
\begin{eqnarray}\label{Star3}
\frac{1}{m}\vfi(x_{i+1})\leq |r_j-r_{j-1}|\leq \frac{1}{8K_1^2K_2}\vfi(x_{i+1}).
\end{eqnarray}

Thus, the collection of subdomains of $\O_i$ to be considered is $\{U_1,\cdots, U_{m-1}\}$ defined by 
\[U_j:=\{(x,y)\in \O_i\,:\,r_{j-1}<x< r_{j+1}\},\] 
for $1\leq j\leq m-1$. The tree structure of the index set $\{1,\cdots,m-1\}$ is the same that we have defined on $\N_0$, where $a=1$ is the root in this case. 
Moreover, we introduce the collection $\{B_j\}_{2\leq j \leq m-1}$ as 
\[B_j:=U_j\cap U_{j-1}=\{(x,y)\in \O_i\,:\,r_{j-1}<x< r_j\}.\] 

Thus, let us prove the hypothesis of Corollary \ref{Rooms and Corridors}. Conditions {\bf (a)}, {\bf (b)} and {\bf (c)} follow easily with $N=2$. From (\ref{Star1}), (\ref{Star2}) and (\ref{Star3}), the measure of any $B_j$ and the whole 
$\O_i$ are comparable to $\vfi(x_{i+1})$. Thus, $\o\geq \frac{1}{M_1}$, where $M_1$ depends only on $K_1,K_2,n$.

Finally, we just have to prove {\bf (e)} for $\ho=1$. Let us show first that each $U_j$, for $1\leq j\leq m-1$, is a star-shaped domain with respect to the ball $\hat{ U}_j$ with center $(r_j,0)$ and radius $(r_{j+1}-r_{j-1})/2$.  
Thus, given $(x,y)\in \hat{U}_j$, $(\tilde{x},\tilde{y})\in U_j$, and $s\in (0,1)$, we have to prove that \[|sy+(1-s)\tilde{y}|<\vfi(sx+(1-s)\tilde{x}).\]

\begin{figure}[htb]
       \center{\includegraphics[width=70mm]{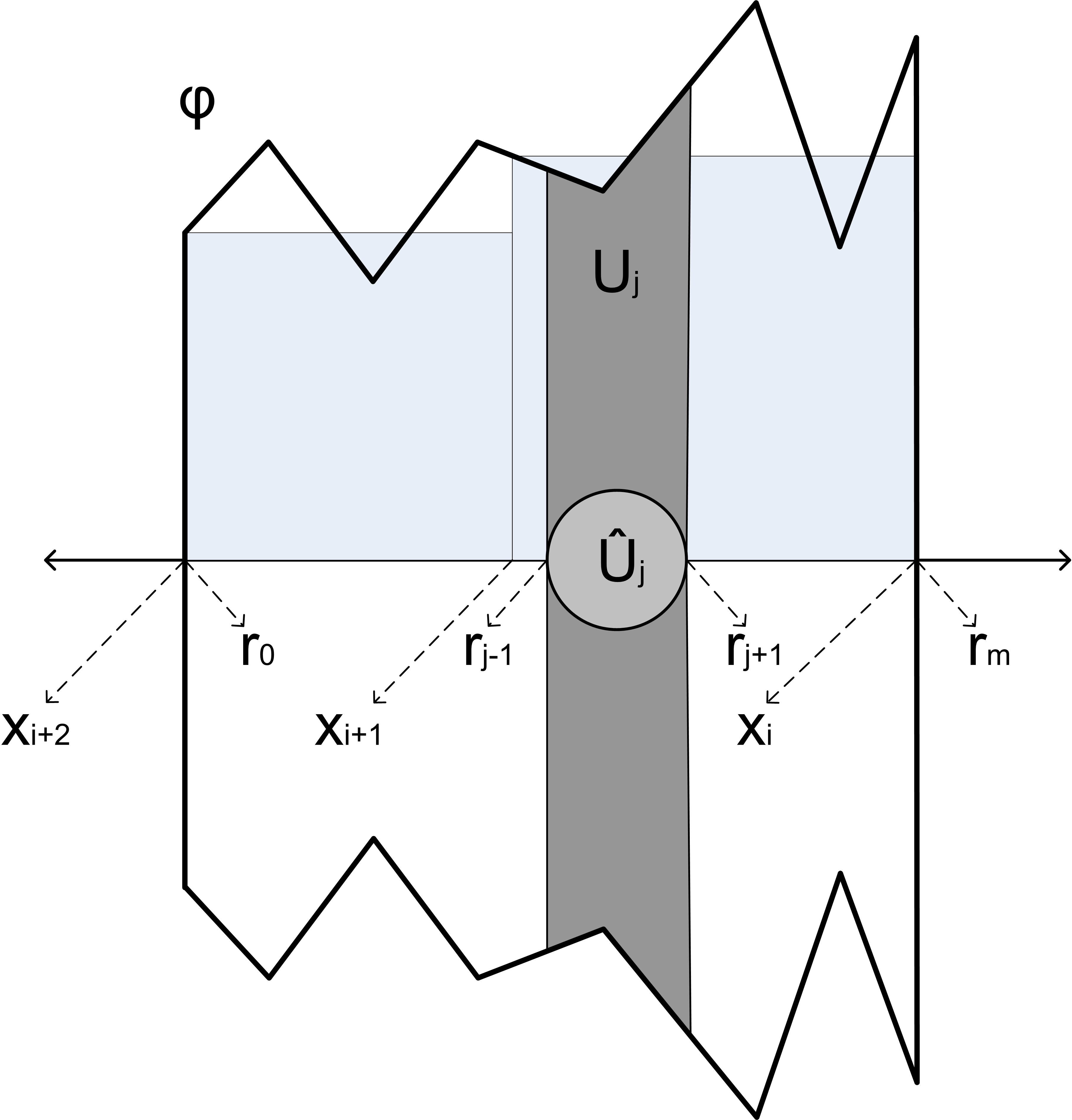}}
        \caption{ A more general $\vfi$ and the star-shaped domain $U_j$}
\end{figure}

To simplify the notation we introduce $M:=\vfi(x_{i+1})/2K_2K_1$.  Thus,
\begin{eqnarray*}
|sy+(1-s)\tilde{y}|&<& s\frac{M}{4K_1}+(1-s)\vfi(\tilde{x})\leq s\left(\frac{M}{4}-M\right)+\vfi(\tilde{x})\\
&\leq& -\frac{3}{4}sM+\left(\vfi(\tilde{x})-\vfi(sx+(1-s)\tilde{x})\right)+\vfi(sx+(1-s)\tilde{x})\\
&\leq& -\frac{3}{4}sM+K_1s|\tilde{x}-x|+\vfi(sx+(1-s)\tilde{x})\\
&\leq& -\frac{3}{4}sM+K_1s\frac{M}{2K_1}+\vfi(sx+(1-s)\tilde{x})< \vfi(sx+(1-s)\tilde{x}).
\end{eqnarray*}
Now, using (\ref{Star2}) and (\ref{Star3}), it can be seen that the diameter of $U_j$ and the radius of $\hat{U}_j$ are comparable to $\vfi(x_{i+1})$. Thus, using estimate (\ref{Galdi}), we can conclude that $M_2$ in {\bf (e)} can be taken as a constant just depending on $K_1$, $K_2$, $n$ and $p$. 
\end{proof}

\section{Divergence problem and Stokes equations on H\"older domains}\label{Holder}

In this section, we show the existence of a right inverse of the divergence operator and the well-posedness of the Stokes equations on an arbitrary bounded H\"older-$\alpha$ domain $\Omega$, i.e. the boundary of $\Omega$ is locally the graph of a function that verifies $|\vfi(x)-\vfi(x')|\leq K_\vfi|x-x'|^\alpha$, for all $x,x'$.
Thus, we start this section studying a domain in $\R^n$ defined by the graph of a positive H\"older-$\alpha$ function $\vfi:(\frac{-3l}{2},\frac{3l}{2})^{n-1}\to\R$, where $0<\alpha\leq 1$ and $l>0$,
\begin{equation}\label{HolderStar}
\O_\varphi:=\left\{(x,y)\in(-l/2,l/2)^{n-1}\times\R\,:\, 0<y< \varphi(x)\right\}\subset\R^n.
\end{equation}
In addition, we assume that $\vfi\geq 2l$ but $\vfi\not\geq 3l$, and $l\leq 1$. Now, $\O$ is locally as $\O_\vfi$, however the distance to the boundary of $\Omega$ is not necessarily equivalent to the distance to the graph of $\vfi$ defined over $(-l/2,l/2)^{n-1}$. Thus, in order to solve this problem, we assume that $\O$ is locally an expanded version of $\O_\vfi$:
\begin{equation}\label{HolderStarExpanded}
\O_{\varphi,E}:=\{(x,y)\in(-3l/2,3l/2)^{n-1}\times\R\,:\, y< \varphi(x)\}\subset\R^n.
\end{equation}
With this new approach of the problem, the distance to $\partial \O$ is equivalent to the distance to $G$ over $\O_\vfi$, where
\begin{equation}\label{HolderGraph}
G:=\{(x,y)\in(-5l/2,5l/2)^{n-1}\times\R\,:\, y= \varphi(x)\}.
\end{equation}

Let us denote the distance to $G$ as $d_G$.

\begin{lemma}\label{HolderStarDiv}
Let $\O_\vfi$ be the domains defined on (\ref{HolderStar}), $1<p<\infty$, and $\kappa\geq 0$.  Given $f\in L^p(\O_\vfi,d_G^{-\kappa})$ with vanishing mean value, there exists a vector field $\uu\in W^{1,p}_0(\O_\vfi,d_G^{1-\alpha-\kappa})^n$  solution of $\di\uu=f$ such that 
\begin{eqnarray*}
\|D\uu\|_{L^p(\O_\vfi,d_G^{1-\alpha-\kappa})}\leq C \|f\|_{L^p(\O_\vfi,d_G^{-\kappa})},
\end{eqnarray*}
where $C$ depends only on $K_\vfi$, $\alpha$, $n$, $p$ and $\kappa$. 
\end{lemma}
\begin{proof} We start defining a collection of cubes in the style of the Whitney cubes, but  in this case the diameter of the cubes is comparable to $d_G$ instead of the distance to $\partial\O_\vfi$. The construction of this collection of open cubes $\{Q_t\}_{t\in\Gamma}$ consists on piling cubes as boxes one over the other one in such a way that the common length of the sides of each cube is comparable to $d_G$. The cubes are constructed level by level starting by level 0, which has just the cube $Q_a=(\frac{-l}{2},\frac{l}{2})^{n-1}\times (0,l)$. The construction of the cubes induces the tree structure of the index set $\Gamma$ where the parents of the index of the cubes in level $m+1$ are the index of the cubes in level $m$. Thus, suppose that we have defined all the cubes in level $m$, and let $Q_t=Q_t'\times(y_{t,1},y_{t,2})$ be one of them. Let us denote by $l_t$ the common length of the sides of $Q_t$. Thus, all the cubes $Q_s$'s in level $m+1$ with $s_p=t$ are defined in the following way: we move up and then expand $Q_t$ to obtain $Q=3(Q_t+(0,\cdots,0,l_t))$ thus
\begin{enumerate}
\item[(i)] if  $Q\subset\O_{\vfi,E}$ there is just one cube $Q_s$ on level $m+1$ such that $s_p=t$ and it is $Q_s=Q_t+(0,\cdots,0,l_t)$,
\item[(ii)] if $Q\not\subset\O_{\vfi,E}$ there are $2^{n-1}$ cubes $Q_s$ with $s_p=t$ and they are written as $Q_s=Q'_{t,1/2}\times(y_{t,2},y_{t,2}+l_t/2)$, where $Q'_{t,1/2}$ is one of cubes in $\R^{n-1}$ obtained by splitting $Q_t'$ into $2^{n-1}$ cubes with length of its sides equal to  $l_t/2$. 
\end{enumerate}
See figure 3 for an example of the construction.

\begin{figure}[htb]
      \center{\includegraphics[width=80mm]{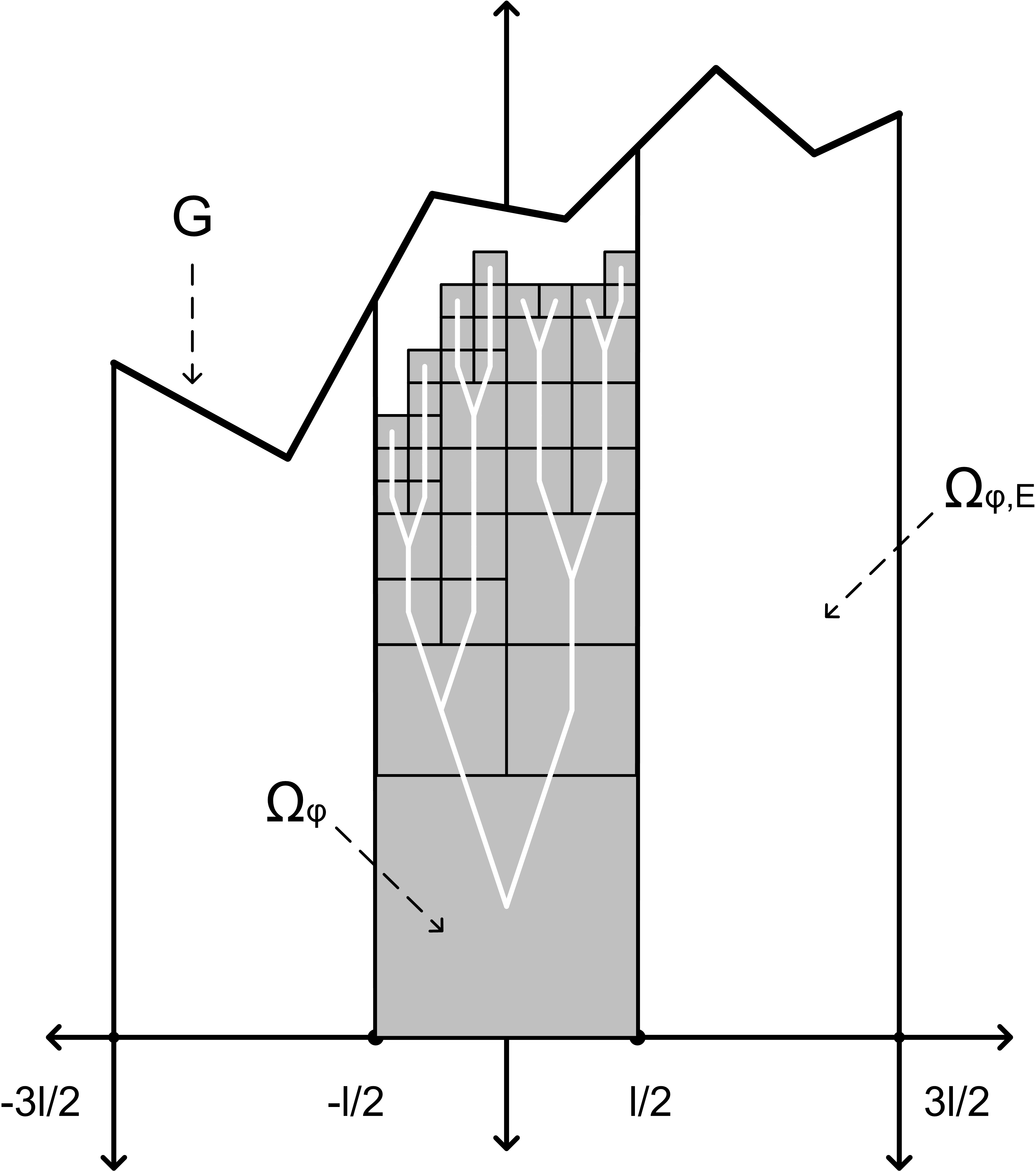}}
      \label{figure3}\caption{ Under the graph of a H\"older-$\alpha$ function}
\end{figure}

Note that the common length of the sides of the cubes $\{Q_t\}_{t\in \Gamma}$ decreases with respect to the order $\succeq$ inherited from the tree. Indeed, 
\begin{eqnarray}\label{Decrease}
l_t\leq l_s \hspace{2cm}\text{ if } t\succeq s.
\end{eqnarray}
Let us show another property satisfied by this collection. Given $Q_t$, it satisfies 
\begin{eqnarray}\label{DistanceCubes}
l_t\leq dist(Q_t,G)\leq C_n\,l_t.
\end{eqnarray}
The first inequality follows from the definition. To prove the second one, note that there exists $Q_s$ with $s\preceq t$ and $l_s=2l_t$ such that $3(Q_s+(0,\cdots,0,l_s))\not\subset\O_{\vfi,E}$. Thus, following the notation 
$Q_s=Q_s'\times(y_{s,1},y_{s,2})$ and $Q_t=Q_t'\times(y_{t,1},y_{t,2})$, we can assert that there exists $x_s\in 3Q_s'$ such that $\vfi(x_s)<y_{s,2}+2l_s\leq y_{t,1}+4l_t.$

On the other hand, $\vfi\geq y_{t,1}+2l_t$ in $Q_t'$, and $Q_t'\subseteq Q'_s$. Thus, using the convexity of $3Q'_s$ and the continuity of $\vfi$, there is another point in $3Q_s'$, let us represent it with the same notation $x_s$, such that 
\begin{eqnarray}\label{Cerca}
y_{t,1}+2l_t\leq\vfi(x_s)\leq y_{t,1}+4l_t. 
\end{eqnarray}

Observe that any point in $Q'_t$ has a distance to $x_s$ lesser than the diameter of $3Q'_s$, which equals $6\sqrt{n-1}\,l_t$. Thus, 

\begin{eqnarray*}
dist(Q_t,G)\leq dist(Q_t,(x_s,\vfi(x_s)))\leq\sqrt{diam(3Q'_s)^2+(3l_t)^2}\leq 3\sqrt{4n-3}\,l_t.
\end{eqnarray*}

Now, we are ready to define the collection $\{\O_t\}_{t\in\Gamma}$ of subdomains of $\O_\vfi$ to apply Theorem \ref{Div}. The first subdomain $\O_a$ is the cube $Q_a$, and the other ones are the $n$ dimensional rectangles defined by  
\begin{eqnarray}\label{NDimensionalRectangles}
\O_t:=Q'_t\times (y_{t,1}-\frac{1}{2}l_t,y_{t,2}),
\end{eqnarray} 
where $Q_t=Q_t'\times(y_{t,1},y_{t,2})$. Observe that $y_{t,2}=y_{t,1}+l_t$.

It can be seen that conditions {\bf (a)} and {\bf (c)} are valid, with a constant $N=2$. Indeed, $\O_t\cap\O_s\neq\emptyset$ for $t\neq s$ if and only if one index is the parent of the other one. 

Next, if we define the collection $\{B_t\}_{t\neq a}$ as 
\[B_t:=\O_t\cap \O_{t_p}=Q'_t\times (y_{t,1}-\frac{1}{2}l_t,y_{t,1}),\] 
{\bf (b)} follows.   

Now, let us show that {\bf (d)} is verified with $\bo:=d_G^{1-\alpha}$. In order to estimate $\o$, we take the point $x_s\in\R^{n-1}$ that verifies (\ref{Cerca}). Then , for all $x_t\in Q'_t$ it follows that  
\begin{eqnarray*}
|\vfi(x_t)| \leq |\vfi(x_t)-\vfi(x_s)|+|\vfi(x_s)|\leq C_{n,\alpha} K_\vfi l_t^\alpha + 4l_t + y_{t,1}.
\end{eqnarray*}

Thus, $|W_t|\leq C_{n,\alpha} (K_\vfi+l_t^{1-\alpha}) l_t^{n-1+\alpha}$ and $|B_t|=\frac{1}{2}l^n_t$. Then, given $s\in\Gamma$ with $s_p=t$, we have that $|W_s|\leq |W_t|$ and $|B_s|\geq 2^{-n}|B_t|$. Thus,  
\begin{eqnarray*}
\omega(x,y) \geq \dfrac{C_{n,\alpha}}{(K_\vfi+l^{n-1+\alpha})}l_t^{1-\alpha}= \dfrac{C_{n,\alpha}}{(K_\vfi+l^{n-1+\alpha})}d_G^{1-\alpha},
\end{eqnarray*}
if $(x,y)$ belong to $\O_t$. Then,
\begin{eqnarray*}
\essinf_{(x,y)\in\O_t}\omega(x,y) \geq \dfrac{C_{n,\alpha}}{(K_\vfi+l^{n-1+\alpha})}l_t^{1-\alpha}\geq \esssup_{(x,y)\in\O_t}\dfrac{C_{n,\alpha}}{(K_\vfi+l^{n-1+\alpha})}d_G^{1-\alpha}(x,y),
\end{eqnarray*}
proving {\bf (d)} with a constant $M_1$ depending only on $K_\vfi$, $n$, $\alpha$ and $l$.

In order to study  {\bf (e)} we take $\ho:=d_G^{-\kappa}$. Using that $\kappa\geq 0$ and $d_G$ is bounded over $\O_\vfi$, it follows that $L^p(\O_\vfi,d_G^{-\kappa})\subset L^1(\O_\vfi)$. Next, from (\ref{DistanceCubes}) we have that $\frac{1}{2}l_t\leq dist(\O_t,G)\leq C_n\,l_t$, hence we can assume that $d_G^{1-\alpha}$ is constant over $\O_t$ reducing the problem to the case $\kappa=0$. Now,  $\O_t$, with $t\neq a$, is a translate of $(0,l_t)^{n-1}\times (0,3l_t/2)$, thus, using Lemma \ref{rectangles}, we can assert that $\O_t$ satisfies $\divp$ with a constant depending only on $n$.

Finally, the condition {\bf (f)} follows from (\ref{Decrease}) and Lemma \ref{Ttreecont}. Moreover, it can be observed that $d_G\leq C_{n}l_s$ in $\O_s$ and $l_s\leq l_t$ if $s\succeq t$, hence $d_G\leq C_{n}l_t$ over $W_t$.

Now, given $g\in L^p(\O_\vfi,d_G^{-\k})$ the function $g d_G^{-\k}$ belongs to $L^p(\O_\vfi)$, and using Lemma \ref{Ttreecont}, we have
\begin{eqnarray*}
\int_\O |Tg|^p d_G^{-p\k}&=&\sum_{t\neq a}\int_{B_t} d_G^{-p\k}\left(\dfrac{1}{|W_t|}\int_{W_t}|g|d_G^{-\k}d_G^{\k}\right)^p\\
&\leq& C_{n,p,\kappa}\sum_{t\neq a}\int_{B_t}  d_G^{-p\k}l_t^{p\kappa}\left(\dfrac{1}{|W_t|}\int_{W_t}|g|d_G^{-\k}\right)^p\\
&=& C_{n,p,\kappa} \sum_{t\neq a} \int_{B_t} \left(\dfrac{1}{|W_t|}\int_{W_t}|g|d_G^{-\k}\right)^p\leq c_{n,p,\kappa} \|gd_G^{-\k}\|_{L^p(\O_\vfi)},
\end{eqnarray*}
proving the lemma.
\end{proof}

\begin{theorem}[Divergence on H\"older-$\alpha$ domains]\label{Divergence in Holder} Let $\O\subset\R^n$ be a bounded H\"older-$\alpha$ domain, and $1<p<\infty$.  Given $f\in L^p_0(\O,d^{-\kappa})$, with $d$ the distance to $\partial\O$ and $\kappa\geq 0$, there exists a vector field $\uu\in W^{1,p}_0(\O,d^{1-\alpha-\kappa})^n$  solution of $\di\uu=f$ with estimate 
\begin{eqnarray*}
\|D\uu\|_{L^p(\O,d^{1-\alpha-\kappa})}\leq C \|f\|_{L^p(\O,d^{-\kappa})},
\end{eqnarray*}
where $C$ does not depend on $f$.
\end{theorem}
\begin{proof} $\O$ is a H\"older-$\alpha$ domain, thus $\partial\O$ is locally the graph of a H\"older-$\alpha$ function, after taking a rigid movement. In fact, we can assume that $\partial\O$ can be covered by a finite collection of open sets $\{U_i\}_{1\leq i\leq m}$ such that $\O_i:=U_i\cap \O$ is in the form (\ref{HolderStar}), where the extended domain in (\ref{HolderStarExpanded}) is the intersection of another open set $V_i\supset U_i$ with $\O$. The reason to consider these $V_i'$s  is just to have $d$ locally comparable to $d_G$. Also, it can be assumed that the finite collection $\{\O_i\}$ is minimal in the sense that for each $1\leq i\leq m$ the set $\O_i\setminus\bigcup_{j\neq i}U_j$ has a positive Lebesgue measure. Now, let us take a Lipschitz domain $\O_0\subset\subset\O$ such that $B_i:=(\O_i\cap U_0)\setminus\bigcup_{j\neq i}U_j$ has a positive Lebesgue measure and $\cup_{i=0}^m\O_i=\O$. 

 Let us define the finite collection $\{\O_i\}_{0\leq i \leq m}$. The tree structure of the index set $\Gamma=\{0,1,\cdots,m\}$ is defined in such a way that two nodes $i$ and $j$ are connected by an edge if and only if one of those is the root $a=0$. Thus, the partial order is given by $i\preceq j$ if and only if $i=0$.

The proof of this theorem follows the idea used to prove Theorem \ref{Div} with a minor difference in condition {\bf (e)}. In this case, the condition has two different weights and it can be stated as: 
given $g\in L^p(\O_t,d^{-\kappa})$ with vanishing mean value there is a solution $\vv\in W_0^p(\O_t,d^{1-\alpha-\kappa})^n$ of $\di\vv=g$ with   
\begin{eqnarray*}
\|D\vv\|_{L^p(\O_t,d^{-\kappa})}\leq M_2 \|g\|_{L^p(\O_t,d^{1-\alpha-\kappa})}
\end{eqnarray*}
for all $t\in\Gamma$, where the positive constant $M_2$ does not depend on $t$. This condition was proved in Lemma \ref{HolderStarDiv}.

Now, note that $\o$ takes finite different positive values, thus the weight $\bo$ can be assumed constant over $\O$ obtaining {\bf (d)}. On the other hand, the operator $T$ has its support in $\O_0$, which is compactly contained in $\O$. Thus, the weight $\ho:=d^{1-\alpha-\kappa}$ can be assumed constant. Hence, from \ref{Ttreecont}, the operator $T:L^p(\O,d^{1-\alpha-\kappa})\to L^p(\O,d^{1-\alpha-\kappa})$ is continuous and {\bf (f)} is satisfied. 

It can be noted that condition {\bf (a)}, {\bf (b)} and {\bf (c)} can be easily verified from the definition of the collection of subdomains and it finiteness. Thus, the proof goes as in Theorem \ref{Div}.
\end{proof}

In the last theorem we show the well-posedness of the Stokes equations on bounded H\"older-$\alpha$ domains.

\begin{theorem}[Stokes on H\"older-$\alpha$ domains]\label{Stokes on Holder} 
Given $\O\subset\R^n$ a bounded H\"older-$\alpha$ domain, and two functions $h\in L_0^2(\O,d^{\alpha-1})$, and $\gg\in H^{-1}(\O)^n$. Then, there exists a unique solution $(\uu,p)\in H_0^1(\O)^n\times L^2(\O,d^{1-\alpha})$ of (\ref{Stokes}) with $\int_\O pd^{2(1-\alpha)}=0$. Moreover, 
\begin{eqnarray}\label{Stokes Estimate 2}
\|D\uu\|_{L^2(\O)}+\|p\|_{L^2(\O,d^{1-\alpha})}\leq C \left(\|\gg\|_{H^{-1}(\O)}+\|h\|_{L^2(\O,d^{\alpha-1})}\right),
\end{eqnarray} 
where $d$ is the distance to $\partial\O$, and $C$ depends only on $\O$.
\end{theorem}
\begin{proof} The proof of this theorem follows the same idea as the proof of Theorem \ref{Stokes in Cuspidal}.
\end{proof}

\end{document}